\documentclass[11pt,a4paper,leqno,twoside]{amsart}%
\usepackage{amsmath,amssymb,amsthm,enumerate,hyperref,color}
\numberwithin{equation}{section}
\newtheorem{theorem}{Theorem}[section]

\newtheorem{proposition}[theorem]{Proposition}
\newtheorem{lemma}[theorem]{Lemma}
\newtheorem{corollary}[theorem]{Corollary}
\newtheorem{remark}[theorem]{Remark}

\newcommand{\rad}{{\text{\upshape rad}}}

\newcommand{\nod}{{\text{\upshape nod}}}

\renewcommand{\a}{\alpha}

\def\e{{\varepsilon}}
\def\g{{\gamma}}

\def\L{{\Lambda}}
\def\l{{\lambda}}
\def\a{{\alpha}}
\def\b{{\beta}}

\newcommand{\jcal}{{\mathcal J}}
\newcommand{\n}{N}
\newcommand{\R}{{\mathbb R}}
\newcommand{\N}{{\mathbb N}}

\def\sideremark#1{\ifvmode\leavevmode\fi\vadjust{\vbox to0pt{\vss
 \hbox to 0pt{\hskip\hsize\hskip1em
 \vbox{\hsize2.1cm\tiny\raggedright\pretolerance10000
  \noindent #1\hfill}\hss}\vbox to15pt{\vfil}\vss}}}%
\newcommand{\edz}[1]{\sideremark{#1}}

\definecolor{darkgreen}{rgb}{0.0, 0.5, 0.2}
\definecolor{purple}{rgb}{0.5, 0.0, 0.5}
\newcommand{\AL}{\color{purple}}

\newcommand{\taglia}{\color{cyan}}

\newif\ifcomment \commentfalse
\def\commentON{\commenttrue}

\long\outer\def\BC#1\EC{\ifcomment \sloppy \par \# \ldots\dotfill
	{\em #1} \dotfill \# \par \fi } \commentON

\newcommand{\remove}[1]{}

\title{Global bifurcation for the H\'enon problem}
\author[A.~L.~Amadori]{Anna Lisa Amadori$^\dag$}
\thanks{The author is member of the Gruppo Nazionale per l'Analisi Matematica, la Probabilit\`a e le loro Applicazioni (GNAMPA) of the Istituto Nazionale di Alta Matematica (INdAM). }
\date{\today}
\address{$\dag$ Dipartimento di Scienze Applicate, Universit\`a di Napoli ``Parthenope", Centro Direzionale di Napoli, Isola C4, 80143 Napoli, Italy. \texttt{annalisa.amadori@uniparthenope.it}}

\begin{document}

\begin{abstract}
We prove the existence of nonradial solutions for the H\'enon equation in the ball with any given number of nodal zones, for arbitrary values of the exponent $\a$. For sign-changing solutions the case $\a=0$ (i.e. the Lane-Emden equation) is included. The obtained solutions form global continua which branch off from the curve of radial solutions $p\mapsto u_p$, and the number of branching points increases with both the number of nodal zones and the exponent $\a$. 
	The proof technique relies on the index of fixed points in cones and provides informations on the symmetry properties of the bifurcating solutions and on the possible intersection and/or overlapping between different branches, thus allowing to separate them at least in some cases.
	
	\
	
	
	
\end{abstract}

\maketitle 

\section{Introduction}
The H\'enon problem, introduced in the 70's for the study of star clusters, see \cite{H}, is
\begin{equation} \label{H}
\left\{\begin{array}{ll}
-\Delta u = |x|^{\alpha}|u|^{p-1} u \qquad & \text{ in } B, \\
u= 0 & \text{ on } \partial B,
\end{array} \right.
\end{equation} 
where $B$ stands for the unitary ball in $\R^N$ with $N\ge 2$ and the exponent $\a$ is positive.
Here we have written the power-type nonlinearity in its odd formulation since we are interested in both positive and sign-changing solutions.
For $\a=0$ \eqref{H} gives back the Lane-Emden problem 
\begin{equation} \label{LE}
\left\{\begin{array}{ll}
-\Delta u = |u|^{p-1} u \qquad & \text{ in } B, \\
u= 0 & \text{ on } \partial B.
\end{array} \right.
\end{equation} 
Some of the results we present here are new also for the latter, and since our techniques allow to deal with both problems simultaneously we shall include the case $\a=0$ in the reasoning.

It is well known that, for $\a>0$ fixed, the H\'enon problem \eqref{H} admits solutions, and in particular radial solutions, for every $p\in(1,p_{\a})$, being 
\[ p_{\a}=\begin{cases} \infty & \text{ in dimension } N=2, \\ \frac{N+2+2\a}{N-2} & \text{ in dimension } N\ge 3 .\end{cases} \]
The same holds when $\a=0$, i.e. for the Lane-Emden equation \eqref{LE}, and in this case the threshold exponent for the existence of solutions coincides with the critical Sobolev exponent $p_0=\frac{N+2}{N-2}$ in dimension $N\ge 3$.
In that range of existence, for any given $m\ge 1$ there  is exactly one couple of radial solutions of \eqref{H} which have exactly $m$ nodal zones, they are classical solutions and they are one the opposite of the other (see \cite{Ni, BWi, NN}, or also \cite{AG-sing-2}). 

Such radial solutions are the only possible ones  only in the framework of positive solutions and Lane-Emden equation, where the celebrated symmetry result by Gidas, Ni and Niremberg \cite{GNN} holds. 
It is well known that the H\'enon problem in the ball has also nonradial positive solutions, and the literature on this subject is rich. First  \cite{SSW} showed that the minimal energy solution is nonradial when $\a$ is large and $p$ is subcritical. After  multi-peak solutions have been constructed by finite-dimensional reduction methods under various incidental assumptions, we mention \cite{EPW, P, PS, HCZ} among others.
Nonradial solutions have also been produced by variational methods as in \cite{S, BS, AG-N=2}, after imposing some constrains on the symmetries of the solutions, and by bifurcation methods  in  \cite{AG14, FN17}.

Coming to nodal solutions, considerations based on the Morse index yield that the minimal energy solution is nonradial for every $\a\ge 0$. Indeed the minimal energy nodal solution has Morse index 2 by \cite{BW}, while the Morse index of nodal radial solutions is greater, see \cite{AP, AG-sing-2}. 
Sign-changing multi-bubble solutions have been produced by finite-dimensional reduction methods,  we can quote  \cite{BMP, EMP, BDP} for the Lane-Emden problem and \cite{ZY} for the H\'enon problem in the disk. They  are very different from the radial ones since their nodal surfaces intersect the boundary of the ball.
Another interesting paper by Gladiali and Ianni \cite{GI} showed the existence of solutions to  the Lane-Emden equation which are nonradial but \textquotedblleft  quasi-radial\textquotedblright, in the sense that their nodal lines are the boundary of nested domains contained in the disc. Some of these quasi-radial solutions are produced as least energy nodal solutions in symmetric spaces, some others by bifurcation w.r.t.~the parameter $p$.
The approach of least energy solutions in symmetric space has been extended also to the H\'enon equation in \cite{AG-N=2, Ama}, always in dimension $N=2$. 
Concerning the H\'enon equation in dimension $N\ge 3$, in the subcritical case a very recent paper by K\"ubler and Weth \cite{KW}  produced  an infinite number of nonradial solutions   by bifurcation w.r.t.~the parameter $\a$, by a fine description of the profile of the radial solutions and of the distribution of their negative eigenvalues  as $\a\to\infty$. Such nonradial  solutions are called by the authors \textquotedblleft almost radial\textquotedblright\ because their nodal surfaces are homeomorphic to spheres. Of course, also the solutions produced by bifurcation arguments in \cite{GI} are of the same kind.

Here we aim to obtain nonradial bifurcation w.r.t. the parameter $p\in (1,p_{\a})$, for any given value of $\a>0$ (and also $\a=0$, as far as sign-changing solutions are concerned),
so we must take into account also the supercritical case.
The  Morse index of radial solutions when the parameter  $p$ approaches  the supremum of the existence range has been recently computed in four different papers (\cite{DIP-N=2, DIP-N>3} concerning the Lane-Emden problem in dimension $N=2$ and $N\ge 3$ respectively, and \cite{AG-N=2, AG-N>3} for the H\'enon problem), while  when $p$ is close to $1$ it has been characterized in terms of the zeros of suitable Bessels function in \cite{Ama}.
Starting from these computations we see that for the positive solution to the H\'enon equation  the Morse index for $p$ close to $1$ is lower than at the supremum of the existence range, and the same holds for nodal solutions  in dimension $N=2$, while in dimension $N\ge 3$ the inequality is reversed. 
Although there are still nontrivial difficulties in deducing actual bifurcation:  no variational structure
can be used to handle supercritical values of $p$ and only an odd change in the Morse index can produce a bifurcation
result. When dealing with the positive solutions, the first  eigenvalue alone plays a role and this ensures that the kernel of the linearized operator contains exactly a one-dimensional subspace of the $O(N\!-\!1)$-invariant functions, and this observation was crucial in both \cite{AG14} and \cite{FN17}.
For nodal solutions, instead, the structure of the kernel  is highly nontrivial. 
We handle this situation by  turning to the notion of degree and index of fixed points in cones introduced by Dancer in \cite{D83}. This approach  has already been applied to the Lane-Emden problem  in an annulus, see  \cite{D92}, and then extended to higher dimension and to sign-changing solutions in \cite{AG-bif}. 
It can be applied also to the H\'enon equation because the exact computations in \cite{AG-N=2, AG-N>3, Ama} rely on a characterization of the Morse index in terms of a singular Sturm-Liouville problem from \cite{AG-sing-1}, which allows to describe in full details the kernel of the linearized operator.
Furthermore this tool provides a detailed bifurcation analysis also for positive solutions, and in the subcritical case,  since it gives informations about the symmetries of the bifurcating solutions and the global properties of the branches.

\

This paper is organized as follows. In Section \ref{sec:stat} we outline the positive cones that we will use and  the main  bifurcation results that we are going to prove. Section \ref{sec:prel} deals with the Morse index: after recalling its characterization by means of the singular eigenvalues and the exact computations performed in the aforementioned papers, we check that the Morse changes across the range $p\in (1,p_{\a})$.
Next in Section \ref{sec:bif} the main results are proved, by taking advantage of the previous discussion on the Morse index and adapting that arguments to compute the index of fixed points in cones. 

\section{Statement of the main results}\label{sec:stat}

We adopt the 	spherical coordinates in $\R^{\n}$ given by $(r,\theta,\varphi)$  with $r=|x|\in[0,+\infty)$, $\theta\in[-\pi,\pi]$, $\varphi =(\varphi_1,\dots\varphi_{\n\!-\!2})\in (0,\pi)^{\n-2}$ 	so that
\[\begin{array}{ll}
x_1=r \cos\theta \prod\limits_{h=1}^{\n\!-\!2}\sin\varphi_h  , \qquad &
x_2= r \sin\theta \prod\limits_{h=1}^{\n\!-\!2}\sin\varphi_h  , \\
x_{k}= r \cos \varphi_{k\!-\!2} \prod\limits_{h=k-1}^{\n\!-\!2}\sin\varphi_h \ \mbox{ as } k=3,\dots \n-1 , \quad  &
x_{\n} = r \cos \varphi_{\n\!-\!2}. 
\end{array}\]
In particular for any $x\neq 0$, $(\theta, \varphi)$ are the coordinates of  $x/|x|   \in \mathbb S_{N-1}$.
Next   for any natural number $n$ we introduce the spaces
\begin{align}\label{H1n}
H^1_{0,n} : = &  \big\{u\in    H^1_0(B)  \, : \,  u(r,\theta,\varphi) \hbox{ is even and } {2\pi}/n  \hbox{ periodic  w.r.t. } \theta , \\ \nonumber 
&  \qquad \qquad \qquad \hbox{ for every } r\in (0,1) \text{ and } \varphi \in (0,\pi)^{\n-2} \big\}, \\
\label{Xn}
X_n : = & 	H^1_{0,n} \cap C^{1,\gamma}(B), 
\intertext{and  the positive cones already used in \cite{AG-bif}, i.e.}
\label{Kn}
K_n : =  & \big\{u\in    X_n \, : \,   \hbox{is nonincreasing w.r.t.~ } \theta\in (0,\pi/n), \\ \nonumber
&  \qquad \qquad  \hbox{ for every } r\in (0,1) \text{ and } \varphi \in (0,\pi)^{\n-2}\big\}.
\end{align}
Notice that radial functions belong to $K_n$ for every $n$.
On the other side, only in dimension $N=2$ the intersection between two different cones  reduces to the radial functions alone.
Instead in dimension $N\ge 3$ it contains also nonradial functions that do not depend on the angle $\theta$.

Throughout the paper we will take the exponent $\a$  as fixed and write $\mathcal{S}^m$ for the curve of radial solutions to \eqref{H} with $m$ nodal zones, precisely
\begin{align}\label{Sn}
\mathcal{S}^m = \big\{ (p, u_p) \in (1, p_{\a})\times C^{1,\gamma}(B) \, : & \  u_p \text{ is the radial solution to \eqref{H} } \\ \nonumber 
& \text{with $m$ nodal zones and } u_p(0)>0 \big\} . \end{align}
We will show that a continuum of nonradial solutions in $K_n$ detaches from the curve $\mathcal S^m$, for some integers $n$ depending on the exponent $\a$ and the number of nodal zones $m$. To this aim we introduce the set
 \begin{equation}\label{Sigman}
 \Sigma_n^m = {\mathcal Cl} \big\{ (p, u) \in (1,p_{\a})\times K_n\setminus \mathcal{S}^m  \, : \, u \mbox{ solves \eqref{H}} \big\},
\end{equation}
where the closure is meant according to the natural norm in $(1,p_{\a})\times C^{1,\gamma}(B)$. Remark that  the set $\Sigma^m_n$ contains also the curves of radial functions $\mathcal S^{m'}$ with $m'\neq m$, but of course $\mathcal S^m$ and $\mathcal S^{m'}$ are separated. So we say that a couple $(p_n, u_{p_n}) \in \mathcal S^m \cap  \Sigma_n^m $   is  a nonradial  bifurcation point, meaning that in  every neighborhood of $(p_n,u_{p_n})$  in the product space $(1,p_{\a})\times C^{1,\g}_0( B )$ there exists a couple $(q,v)$ such that $v$ is a nonradial solution of \eqref{H} related to the exponent $q$. 
In this case we set
\begin{equation}\label{cn}
\mathcal C^m_n \ \mbox{ the closed connected component of $\Sigma^m_n$ containing $(p_n,u_{p_n})$}
\end{equation}and we shall refer it as the \textquotedblleft branch\textquotedblright  departing from $(p_n, u_{p_n}) $, with a little misuse of language.
We will also write $[t]$ and $\lceil t\rceil$, respectively, for the floor and the ceiling of a real number $t$, i.e.
\[ [t]=\max \left\{ n\in\mathbb Z \, : \, n \le t\right\} , \quad \lceil t\rceil= \min\left\{n\in\mathbb Z \, : \, n \ge t \right\} .\]
Eventually the same reasoning enables us to prove several bifurcation results. First we produce $\lceil\frac{\a}{2}\rceil$ global branches of positive nonradial solutions, precisely

\begin{theorem}[Bifurcation from positive solutions]\label{teo:bif-H-1}
			In any dimension $N\ge 2$ and for every $\a>0$, there are at least $\lceil\frac{\a}{2}\rceil$ different points along the curve $\mathcal S^1$ where a nonradial bifurcation occurs.
			More precisely for every $n=1, \dots \lceil\frac{\a}{2}\rceil$ there exists a nonradial bifurcation point $(p_n,u_{p_n})\in \mathcal S^1 \cap \Sigma_n^1$ and the respective branch $\mathcal C^1_n$ has the following global properties
			\begin{enumerate}[i)]
				\item $\mathcal C^1_n$ is made up of positive solutions and unbounded, i.e. it contains a sequence $(p_k, u_k)$ with  $\|u_k\|_{C^{1,\gamma}}\to \infty$ or   $p_k\to p_{\a}$.
				\item In dimension $N=2$ the branches are separated, in the sense that their intersection contains at most isolated points along the curve of positive radial solutions $\mathcal S^1$.
				\item	In dimension $N\ge 3$ two different branches can only have in common couples $(p, v)$, where $v$ are positive solutions to \eqref{H} which	do not depend on the angle $\theta$, and their overlapping  can even make up a  continuum. 	
			\end{enumerate}
\end{theorem}

In the disc solutions enjoying the same symmetry properties 
have been produced in \cite{EPW} 
by the Lyapunov-Schmidt reduction method, and in \cite{AG-N=2} by minimizing the energy associated to \eqref{H} in the space $H^1_{0,n}$. In this last paper it has been proved that such \textquotedblleft least energy $n$-invariant  solutions\textquotedblright are nonradial and different one from another at least for $p\in (p_n,+\infty)$, with $p_n$ the same exponent appearing here. On the other hand, they are certainly radial for $p$ close to one, thanks to the uniqueness result in \cite{AG-bif}. It is therefore natural to think that  the branches of bifurcating solutions shown by Theorem \ref{teo:bif-H-1} are made up by these  least energy $n$-invariant  solutions, and so they do exist  for every $p\in (p_n, \infty)$, and are separated. 
\\
In higher dimension Theorem \ref{teo:bif-H-1} improves the bifurcation result obtained in \cite{AG14}, which holds for $\a\in(0,1]$ and produces only one branch of nonradial solutions.
Nonradial solutions with similar symmetries have been produced by the finite-dimensional reduction method: in particular \cite{PS} concerns the slightly subcritical case and exhibits  solutions which blow up when $p$ approaches the critical Sobolev exponent,  while \cite{HCZ} proves the existence  also in the critical case.  
Besides nonradial solutions  do exist also for $p$ close to $p_{\a}$, as showed in \cite{FN17}. It is very likely that some of the nonradial solutions found in Theorem \ref{teo:bif-H-1} coincide with the ones in \cite{FN17}, where the specular viewpoint (bifurcation w.r.t. $\a$) is adopted.

\

Coming to nodal solutions, the asymptotic Morse index and consequently  the number of nonradial branches depend on the dimension. We therefore state the bifurcation results separately.
\\
In the plane the set $\Sigma^2_n$ is nonempty at least for $n=\left[\frac{2+\a}{2}\beta+1\right], \dots \left\lceil\frac{2+\a}{2}\kappa-1\right\rceil$, where  $\beta\approx 2,\!305$ and $\kappa  \approx 5,\!1869 $ are fixed numbers  related to the computation of the Morse index at $p$ next to 1 and at infinity, respectively, whose characterization is recalled in Section \ref{sec:prel}.  Precisely we have

\begin{theorem}[Bifurcation from nodal solutions in dimension $N=2$]\label{teo:bif-H-m-N=2}
		Consider problem \eqref{H} in dimension $N=2$. 
		 For every $\a\ge 0$ there are at least $\left\lceil\frac{2+\a}{2}\kappa-1\right\rceil -\left[\frac{2+\a}{2}\beta\right] $  different points along the curve $\mathcal S^2$ where nonradial bifurcation occurs.
		 More precisely for every $n=\left[\frac{2+\a}{2}\beta+1\right], \dots \left\lceil\frac{2+\a}{2}\kappa-1\right\rceil$ there exists a nonradial bifurcation point $(p_n,u_{p_n})\in \mathcal S^2 \cap \Sigma_n^2$ and the respective branches $\mathcal C^2_n$ have the following properties
	\begin{enumerate}[i)]	
		\item There is a ball $\mathcal B$ in $(1,\infty)\times C^{1,\gamma}(B)$ centered at $(p_n,u_{p_n})$ such that $\mathcal C^2_n\cap \mathcal B\setminus\{(p_n, u_{p_n})\}$ is made up of nonradial solutions with $2$ nodal zones, one of which contains $x=0$ and is  homeomorphic to a disc.
		\item Every branch	contains a sequence $(p_k, u_k)$ with either $\|u_k\|_{C^{1,\gamma}}\to \infty$, or $p_k\to \infty$, or possibly $p_k\to 1$ and $u_k$ converges to an eigenfunction of 
\begin{equation}\label{prima-autof-weight} 
\left\{
\begin{array}{ll}
-\Delta \omega= \mu |x|^{\a} \omega & \text{ in } B, \\
\omega = 0 & \text{ on } \partial B ,  
\end{array}\right.
\end{equation} 	which belongs to $K_n$.	
\item Two different branches can only have radial solutions in common. 
Precisely $\mathcal C^2_n \cap \mathcal C^2_{n'} \cap \mathcal S^2$ contains at most isolated points, and if there is some $m\ge 3$ such that $\mathcal C^2_n \cap \mathcal C^2_{n'} \cap \mathcal S^m$ is nonempty, then $\mathcal S^m \subset \mathcal C^2_n \cap \mathcal C^2_{n'} $.
\end{enumerate}	
\end{theorem}
	
The possibility that $p_k\to 1$ but $u_k$ stays bounded remains open because  the uniqueness of nodal solutions does not hold either in a neighborhood of $p=1$, see \cite[Theorem 1.3]{Ama}. 
Concerning property {\it iii)}, i.e. the possible overlapping of two different branches, we are not aware of any technique which enables to capture the formation of further nodal zones and/or a secondary  bifurcation. Consequently a nonradial branch could, in principle, touch another radial curve $\mathcal S^m$ with $m\ge 3$, and then incorporate it because of the way in which $\Sigma^2_n$ and $\mathcal C^2_n$ have been defined.

Theorem \ref{teo:bif-H-m-N=2} applies also to $\a=0$, i.e.~to the Lane-Emden equation, giving back  \cite[Theorem 1.2]{GI} since in this particular case $\left[\frac{2+\a}{2}\beta+1\right]= 3$ and $\left\lceil\frac{2+\a}{2}\kappa-1\right\rceil=5$.
\\
For $\a>0$ it is worth comparing this existence result with the ones in \cite{Ama} and in \cite{AG-N=2}, both concerning the least energy $n$-invariant  nodal solutions, that we denote hereafter by $U_{p,n}$.
For $n=1,\dots \left\lceil\frac{2+\a}{2}\beta-1\right\rceil$,  $U_{p,n}$ is nonradial for both $p$ close to 1 and  large.
It seems that in this case  $U_{p,n}$ is nonradial for every $p>1$ and the curve $p\mapsto  U_{p,n}$ does not intersect the curve of radial solutions. This is certainly true for $n=1$, i.e. the least energy nodal solution. 
Conversely for  $n=\left[\frac{2+\a}{2}\beta+1\right], \dots \left\lceil \frac{2+\a}{2}\kappa-1\right\rceil$, \cite[Proposition 4.10]{Ama} and \cite[Theorem 1.6]{AG-N=2} yield  that  $U_{p,n}$ are radial for $p$ close to 1, and then nonradial (and different one from another) when $p$ is large. Therefore
the curves  $p\mapsto  U_{p,n}$  coincide with the one of radial solutions for $p\in (1,p_n)$, and then they give rise to the  nonradial bifurcation stated by Theorem \ref{teo:bif-H-m-N=2}. 

Only bifurcation from the curve $\mathcal S^2$ is taken into account, since the behaviour of nodal solutions as $p\to \infty$ is known only in the case of two nodal zones. When this paper was already finished we came to know that a very recent preprint by Ianni and Saldana  \cite{IS} describes the asymptotic profile of every radial solutions. Starting from this it is possible, in principle, to compute exactly their Morse index and then the same arguments used here produce bifurcation also in the general case.

\

In dimension $N\ge 3$ the set $\Sigma^m_n$ is nonempty at least for $n=2+ \big[\frac{\a}{2}\big]$, $\dots$ $n_{\a}^m$,
where the number $n_{\a}^m\ge 2(m-1)+ [\a(m-1)]$ is characterized later on in Remark \ref{n-def} and can be numerically computed.

\begin{theorem}[Bifurcation from nodal solutions in dimension $N\ge 3$]\label{teo:bif-H-m}
	Consider problem \eqref{H} in dimension $N\ge3$.  For every $\a\ge 0$ and $m\ge 2$, at least $2m-3+ [\a(m-1)]- [\a/2]$  different nonradial bifurcations take place along the curve $\mathcal S^m$.
	More precisely for every $n=2 + \big[\frac{\a}{2}\big] , \dots n_{\a}^m $ there exists a nonradial bifurcation point $(p_n,u_{p_n})\in \mathcal S^m \cap \Sigma_n^m$ and the respective branches $\mathcal C^m_n$ have the following properties
\begin{enumerate}[i)]	
	\item There is a ball $\mathcal B$ in $(1,p_{\a})\times C^{1,\gamma}(B)$ centered at $(p_n,u_{p_n})$ such that $\mathcal C^m_n\cap \mathcal B\setminus\{(p_n, u_{p_n})\}$ is made up of nonradial solutions with $m$ nodal zones, one of which contains $x=0$ and is  homeomorphic to a ball, while the other ones are homeomorphic to spherical shells.
	\item Every branch contains a sequence $(p_k, u_k)$ with either $\|u_k\|_{C^{1,\gamma}}\to \infty$,  or $p_k\to p_{\a}$, or possibly $p_k\to 1$ and $u_k$ converges to an eigenfunction of  \eqref{prima-autof-weight} 	which belongs to $K_n$.
	\item The intersection between two different branches, if non-empty, is made up of nodal solutions which do not depend by the angle $\theta$.
\end{enumerate}	
\end{theorem}

The branches of nodal bifurcating solution in dimension $N\ge 3$ can overlap along radial solutions with a different number of nodal zones, but also along nonradial solutions that do not depend by the angle $\theta$.
\\
The statement of Theorem \ref{teo:bif-H-m} is  new also in the simpler case $\a=0$, to the author's knowledge.
For the reader's convenience, we state separately the bifurcation result concerning the Lane-Emden equation. 

 	\begin{theorem}[Bifurcation  for the Lane Emden equation in dimension $N\ge 3$]\label{teo:bif-LE}
		Consider problem \eqref{LE} in dimension $N\ge3$.  For every $m\ge 2$ the curve $\mathcal S^m$ bifurcates at $2m-3$ points, at least.
		More precisely for every $n=2 , \dots n^m_0 $  there exists a nonradial bifurcation point $(p_n,u_{p_n})\in \mathcal S^m \cap \Sigma_n^m$ 
		and   the continuum detaching at  $(p_n, u_{p_n}) $, i.e. $\mathcal C^m_n$ has the following
	\begin{itemize}	
		\item Local property: there is a ball $\mathcal B$ in $(1,p_{0})\times C^{1,\gamma}(B)$ centered at $(p_n,u_{p_n})$ such that $\mathcal C^m_n\cap \mathcal B\setminus\{(p_n, u_{p_n})\}$ is made up of nonradial solutions with $m$ nodal zones, one of which contains $x=0$ and is  homeomorphic to a ball, while the other ones are homeomorphic to spherical shells,
	\item Global property: every branch  contains a sequence $(p_k, u_k)$ with either $\|u_k\|_{C^{1,\gamma}}\to \infty$, or $p_k\to p_{0}$, or possibly $p_k\to 1$ and $u_k$ converges to an eigenfunction of  
		\begin{equation}\label{prima-autof} 
		\left\{
		\begin{array}{ll}
		-\Delta \omega= \mu \, \omega & \text{ in } B, \\
		\omega = 0 & \text{ on } \partial B ,
		\end{array}\right.
		\end{equation}	
		which belongs to $K_n$.
	\item Separation property: the intersection between two different branches, if non-empty, is made up of nodal solutions which do not depend by the angle $\theta$.
	\end{itemize}	
	\end{theorem}
There is numerical evidence that $ n^m_0= 2(m-1)$ in any dimension $N\ge 3$, so that Theorem \ref{teo:bif-LE} provides exactly $2m-3$ branches of nonradial solutions. In particular, in the case of $2$ nodal zones, there should be only one branch in dimension $N\ge 3$, while 3 different branches have been produced in dimension $N=2$.  The planar case indeed differs from the other ones, as already observed in several occasions.

\

Let us mention in passing that the number of nonradial branches produced in Theorems \ref{teo:bif-H-1}, \ref{teo:bif-H-m-N=2} and \ref{teo:bif-H-m} goes to infinity when $\alpha\to \infty$, which is consistent with the specular study (bifurcation w.r.t. $\a$) performed in \cite{KW}.

\section{Preliminaries on the computation of the Morse index}\label{sec:prel}

To emphasize the dependence on the exponent $p\in (1, p_{\a})$, we take the exponent $\a\ge 0$ and the number of nodal zones $m$ as fixed and denote by $u_p$ the unique radial solution to \eqref{H} with $m$ nodal zones which is positive at the origin.
We also write 
\begin{align}
\label{linearized}
L_{p} \psi &=-\Delta \psi-p |x|^\a |u_p|^{p-1}\psi, \\
\label{forma-quadratica}
{\mathcal Q}_p(\psi)& =\int_\Omega \left(|\nabla \psi|^2 -p|x|^\a |u_p|^{p-1}\psi^2\right) dx
\end{align}
for the linearized operator at $u_p$  and the related quadratic form, respectively. They will be considered on the space $H^1_0(B)$, or in one of its subspaces specified case-by-case.
\\
The Morse index, that we denote hereafter by $m(u_p)$, is the maximal dimension of a subspace of $H^1_0(B)$ in which the quadratic form ${\mathcal Q}_p$ 
	is negative defined, or equivalently the number of the negative eigenvalues of  
	\begin{equation}\label{standard-eig-prob}
	L_p \psi = \Lambda \psi , \qquad \psi \in H^1_0(B).
	\end{equation}
For radial solutions one can also look at the radial Morse index, denoted by $m_\rad(u_p)$, i.e. the number of the negative eigenvalues of for \eqref{standard-eig-prob} whose relative eigenfunction is $H^1_{0,\rad}(B)$, the subspace of $H^1_0(B)$ given by radial functions. 
\\
As explained in full details in \cite{AG-sing-1}, this matter can be regarded through a {\it singular} eigenvalue problem associated to the linearized operator $L_p$,  which has to be handled in weighted Lebesgue and Sobolev spaces
\begin{align*}
{\mathcal L} & = \{\omega:  B\to \R\, : \,  \omega/|x|  \in L^2(B)\}, \qquad  \mathcal{H}_0 =H_0^1(B)\cap \mathcal L .
\end{align*}
The Morse index (on $H^1_0(B)$ as well as on some of its  subspaces) turns out to be equal to the number of the negative eigenvalues of  
\begin{equation}\label{singular-eig-prob}
L_p \psi = \widehat\Lambda \psi / {|x|^2}  , \qquad \psi \in {\mathcal H}_0(B).
\end{equation}
Concerning radial solutions to the H\'enon problem, it turns helpful the transformation 
	\begin{equation}\label{transformation-henon}
t=r^{\frac{2+\a}{2}} , \qquad w(t)=u(r) ,
\end{equation} 
introduced in  \cite{GGN}, or a slight variation of it
	\begin{equation}\label{transformation-henon-no-c}
t=r^{\frac{2+\a}{2}} , \qquad v(t)= \left(\frac{2}{2+\a}\right)^{\frac{2}{p-1}} u(r),
\end{equation} 
 which map radial solutions to \eqref{H} into solutions of one-dimensional problems
\begin{equation}\label{LE-radial}
\begin{cases}
- \left(t^{M-1} w^{\prime}\right)^{\prime}= \left(\frac{2}{2+\a}\right)^2 t^{M-1} |w|^{p-1} w  , \qquad  & 0<t< 1, \\
w'(0)=0, \quad w(1)=0
\end{cases}\end{equation}
or, respectively
\begin{equation}\label{LE-radial-no-c}
\begin{cases}
- \left(t^{M-1} v^{\prime}\right)^{\prime}= t^{M-1} |v|^{p-1} v  , \qquad  & 0< t< 1, \\
v'(0)=0, \;  v(1) =0 , & 
\end{cases}\end{equation}
see \cite{AG-sing-1, AG-sing-2}. In both cases $M$ is a real parameter given by 
\begin{align}\label{Malpha}
M & = M(N,\alpha) = \frac{2(N+\alpha)}{2+\alpha}\in[2,N]  .
\end{align} 
Both \eqref{LE-radial} and \eqref{LE-radial-no-c} can be regarded as generalized radial versions of the Lane-Emden problem \eqref{LE}, since for integer values of $M$ the function $v$ solves  indeed a problem of type \eqref{LE} settled in the unit $M$-dimensional ball. In general we should refer to $M$ as the fictitious dimension of the associated generalized Lane-Emden problem. 
The natural generalization of the standard Lebesgue and  Sobolev spaces  from which regarding at \eqref{LE-radial}, or \eqref{LE-radial-no-c} are 
\begin{align*} 
L^q_M& = \{\varphi :(0,1)\to\R\, : \, \varphi \text{ measurable and s.t. } \int_0^1 t^{M-1} |\varphi|^q dt < +\infty\} , \\ H^1_M&  = \{\varphi\in L^2_M \, : \, \text{ $\varphi$ has a first order weak derivative $\varphi'$ in }L^2_M \}, \\  H^1_{0,M}  & = \left\{ \varphi\in H^1_M \, : \, \varphi(1)=0\right\}
\end{align*}
Similarly the weighted Lebesgue and Sobolev spaces related to \eqref{singular-eig-prob} shall be
\begin{align*}
{\mathcal L}_M &  =\{ \varphi: (0,1)\to \R \, : \, \varphi/t \in L^2_M \} , \qquad \mathcal{H}_{0,M} =H^1_{0,M}\cap {\mathcal L}_M .
\end{align*}
Of course ${\mathcal L}_M$ is a Hilbert space with the obvious product which brings to the orthogonality condition
	\[ \varphi \underline{\perp}_M \psi \, \Longleftrightarrow \, \int_0^1 t^{M-3} \varphi \, \psi  \, dt = 0 .\]
In this functional setting it is possible to look at a singular eigenvalue problem associated to \eqref{LE-radial}, or equivalently \eqref{LE-radial-no-c}, that is
\begin{equation}\label{eigenvalue-problem}
\left\{\begin{array}{ll}
- \left(t^{M-1} \phi'\right)'- t^{M-1} a_p(t)  \phi = t^{M-3} {\nu} \, \phi & \text{ for } t\in(0,1)\\
\phi\in  \mathcal H_{0,M} ,
\end{array} \right.
\end{equation} 
where
\begin{equation}\label{ap}
a_p(t):= p  \left(\frac{2}{2+\a}\right)^2 |w_p(t)|^{p-1} = p |v_p(t)|^{p-1} .
\end{equation} 
The definition of the singular eigenvalues $\nu$ requests some care because of the singularity of the Sturm-Liouville problem \eqref{eigenvalue-problem} at the origin. It has been tackled in detail in \cite[Section 3]{AG-sing-1}\footnote{for the reader's convenience, we remark that in the quoted paper the singular eigenvalues were denoted by $\widehat\nu$, while $\nu$ stand for the classical eigenvalues associated with the linearization of \eqref{LE-radial}}, by establishing and exploiting a variational characterization.
Indeed, one can alternatively define
\begin{equation}\label{eigenvalue-variat-1}
\nu_1 := \inf\left\{ \frac{\int_0^1 t^{M-1}\left(|\phi'|^2 - a_p \phi^2\right)dt}{\int_0^1 t^{M-3} \phi^2 dt }: \phi\in \mathcal H_{0,M} , \ \phi\neq 0 \right\}
\end{equation}
and see that, when the infimum stays below the threshold $(M-2)/2$, then it is attained by a function $\phi$ which solves \eqref{eigenvalue-problem} in the weak sense, namely
\begin{equation}\label{eigenvalue-problem-sol}
\int_0^1 t^{M-1} \left( \phi'\varphi'-a_p \phi\varphi\right) dt = \nu_1\int_0^1 t^{M-3} \phi\varphi \, dt 
\end{equation}
for every test function $\varphi\in \mathcal H_{0,M}$.
Such function $\phi$ can therefore be called an eigenfunction related to the eigenvalue $\nu_1$, and denoted by $\phi_1$. Iteratively, if $\nu_i< (M-2)/2$, one can settle the minimization problem 
\begin{equation}\label{eigenvalue-variat}
\nu_{i+1} := \inf\left\{ \frac{\int_0^1 t^{M-1}\left(|\phi'|^2 - a_p \phi^2\right)dt}{\int_0^1 t^{M-3} \phi^2 dt }: \phi\in \mathcal H_{0,M} ,  \ \phi \underline{\perp}_M \phi_1, \dots \phi_i \right\} .
\end{equation}
Again, as far as  $\nu_{i+1} <(M-2)/2$, it is attained by an eigenfunction $\phi_{i+1}$ which solves \eqref{eigenvalue-problem}  in the weak sense.
Next, the eigenfunctions related to these singular eigenvalues enjoy the same properties of the standard ones, in particular they are simple, mutually orthogonal, and the $i^{th}$ eigenfunction has exactly $i$ nodal domains.

Eventually  putting together \cite[Proposition 1.4]{AG-sing-1} and  \cite[Proposition 3.3, Theorem 1.3]{AG-sing-2}  we have 
\begin{proposition}\label{general-morse-formula-H}
	 Let $\a\geq 0$ and $u_p$ be a radial solution to \eqref{H} with $m$ nodal zones.
Then  the only nonnegative eigenvalues of \eqref{eigenvalue-problem} are $\nu_1(p)<\nu_2(p)<\dots< \nu_m(p)<0$ and satisfy
\begin{align}
\label{nui<k} & {\nu}_i(p)  < -\frac{2N-2+\a}{2+\alpha}   & \text{for } i=1,\dots m-1 ,
\\
\label{num>k} & -\frac{2N-2+\a}{2+\alpha} <{\nu}_m(p) <0  .  &
\end{align}
Moreover the Morse index of $u_p$ is given by
	\begin{align}\label{tag-2-H}
	m(u_p) & =\sum\limits_{i=1}^{m}      \sum\limits_{j=0}^{\lceil J_i  -1\rceil} N_j ,
	\end{align}
	\begin{tabular}{ll} 
	where & $\lceil s \rceil = \{\min n\in \mathbb Z \, : \, n\ge s\}$ denotes the ceiling function  and \\
& $J_i(p)=\frac{2+\a}{2} \left(\sqrt{\left(\frac{N-2}{2+\a}\right)^2- \nu_i(p)}-\frac{N-2}{2+\a}\right)$, 
	\end{tabular}

\noindent being 	$	N_j=\frac{(N+2j-2)(N+j-3)!}{(N-2)!j!}$  the multiplicity of the eigenvalue  
	$\l_j=j(N+j-2)$ for the Laplace-Beltrami operator in the sphere ${\mathbb S}_{N\!-\!1}$.
\end{proposition}

Afterward, the Morse index has been exactly computed at the ends of the existence range by computing the limits of the eigenvalues $\nu_i(p)$.
The paper \cite{Ama} dealt with $p$ close to 1, and we need to introduce some more notation to recall the obtained result. For every $\beta\ge 0$ we write $\mathcal J_{\beta}$ for the Bessel function of first kind 
\[\jcal_{\beta}(r) = r^{\beta}\sum\limits_{k=0}^{+\infty} \dfrac{(-1)^k}{k!\Gamma(k+1+\beta)} \left(\frac{r}{2}\right)^{2k}, \quad r\ge 0 , \]
and $z_i(\beta)$ for the sequence of its positive zeros.
Since the map $\beta\mapsto z_i(\beta)$ is continuous and increasing, for every fixed  integer $m$  there exist $\beta_i=\beta_i(\alpha, N,m)$  such that 
\begin{equation}\label{betai-def} \begin{split}
 z_i(\beta_i) \text{ (the $i^{th}$ zero of the Bessel function ${\mathcal J}_{\beta_i}$)} \\
\text{ coincides with $z_m\Big(\frac{\n-2}{2+\alpha}\Big)$ (the $m^{th}$ zero of ${\mathcal J}_{\frac{\n-2}{2+\alpha}}$)}.
 \end{split}\end{equation}
It is clear that 
\[ \beta_1>\beta_2 \dots >\beta_m=\frac{\n-2}{2+\alpha} .\]

In \cite[Proposition 3.3 and Theorem 1.2]{Ama} it is proved that 
\begin{theorem}\label{teo:mi-p=1} 	
	 Let $\a\geq 0$ and $u_p$ be a radial solution to \eqref{H} with $m$ nodal zones.
	 Then 
	 \begin{align}\label{nup=1}
	\lim\limits_{p\to 1} \nu_i(p) = \left(\frac{\n-2}{2+\alpha}\right)^2 -\beta_i^2 \quad \text{ as } i=1,\dots m.
	 \end{align}
	After there exists $\bar p =\bar p(\a)>1$ such that for $p\in(1,\bar p)$ the Morse index of $u_p$ is given by
	\begin{equation}\label{mi-p=1}
	m(u_p)= 1+\sum_{i=1}^{m-1}\sum _{j=0}^{\left\lceil \frac{(2+\a)\beta_i -N}{2} \right\rceil} N_j  
	\end{equation}
	if $\alpha \neq \a_{\ell,n} = (2n+N)/\beta_{\ell} -2$, and it is estimated by 
	\begin{equation}\label{mi-p=1brutta} \begin{split}	1+\sum_{i=1}^{m-1}\sum _{j=0}^{\left\lceil \frac{(2+\a)\beta_i -N}{2} \right\rceil} N_j \le  m(u_p)	\le 1 +\sum_{i=1}^{m-1}\sum _{j=0}^{\left\lceil \frac{(2+\a)\beta_i -N}{2} \right\rceil} N_j + \sum\limits_{\ell} N_{1+\frac{(2+\a)\beta_{\ell} -N}{2}}   .	\end{split}	\end{equation}
	if $\alpha = \a_{\ell,n}$ for some $\ell$ and $n$.
\end{theorem}

The situation at the supremum of the existence range changes drastically depending if the dimension is $N=2$ or greater.
The Morse index in dimension $N\ge 3$ is computed in \cite{AG-N>3} extending some previous results on the Lane-Emden problem in \cite{DIP-N>3}; precisely \cite[Propositions 3.3, 3.10 and Theorem 1]{AG-N>3} state that
\begin{theorem}\label{teo:mi-p=palpha} 	
	Let $\a\geq 0$ and $u_p$ be a radial solution to \eqref{H} with $m$ nodal zones in dimension $N\ge 3$.
	Then 
	\begin{align}\label{nu-p=palpha}
	\lim\limits_{p\to p_{\a}} \nu_i(p) = - \frac{2N-2+\a}{2+\alpha} \quad \text{ as } i=1,\dots m.
	\end{align}
	After there exists $p^{\star}=p^{\star}(\a)\in (1,p_{\a})$ such that  the Morse index of $u_p$ is given by
	\begin{equation}\label{mi-p=palpha}
	m(u_p)= \sum\limits_{j=1}^{\left\lceil \frac{\a}{2}\right\rceil} N_j+ (m-1)\sum _{j=0}^{\left[\frac{2+\a}{2}\right]}  N_j  
	\end{equation}
for $p\in(p^{\star}, p_{\a})$.
\end{theorem}

In dimension $N=2$ only the Morse index of the least energy radial solution (i.e. the positive one) and of the least energy nodal radial solution (i.e. the  one with two nodal zones) are known. They have both been computed in the paper \cite{AG-N=2}, where it is shown that
 
 \begin{theorem}\label{teo:mi-p=infty-1} 	
	Let $\a\geq 0$ and $u_p$ be a positive radial solution to \eqref{H} in dimension $N=2$.
	Then 
 	\begin{align}\label{nu-p=infty-1}
 	\lim\limits_{p\to \infty} \nu_1(p) = - 1 ,
 	\end{align}
 	and there exists $p^{\star} > 1$ such that for $p > p^{\star}$ the Morse index of $u_p$ is given by
 	\begin{equation}\label{mi-p=infty-1}
 	m(u_p)=  1 + 2 \left\lceil \frac{\a}{2}\right\rceil .
 	\end{equation}
 \end{theorem}

 \begin{theorem}\label{teo:mi-p=infty-2} 	
	Let $\a\geq 0$ and $u_p$ be a  radial solution to \eqref{H} with 2 nodal zones in dimension $N=2$, then 
 	\begin{align}\label{nu-p=infty-2}
 		\lim\limits_{p\to \infty} \nu_1(p) = - \kappa^2  \ \text{ with } \kappa\approx 5,\!1869 \qquad 
 		\lim\limits_{p\to \infty} \nu_2(p) = - 1 .
 	\end{align}
 	Moreover there exists $p^{\star} > 1$ such that for $p > p^{\star}$ the Morse index of $u_p$ is given by
 	\begin{align}\label{mi-p=infty-2}
 	& m(u_p)  =2\left\lceil\frac {2+\a}2\kappa\right\rceil +  2\left\lceil\frac{\alpha}{2}\right\rceil 
 	\intertext{when $\a\neq \a'_n= 2(n/{\kappa}-1)$, while when $\a=\a'_n$ it holds} \label{mi-p=infty-2-brutta}
 	(2+\a) \kappa  +  2\left\lceil\frac{\alpha}{2}\right\rceil & \le  m(u_p)\le  (2+\a) \kappa+  2\left\lceil\frac{\alpha}{2}\right\rceil+ 2 .
 	\end{align}
 \end{theorem}
 
An analogous result for the radial solution to the Lane-Emden problem  with two nodal zones has been obtained in \cite{DIP-N=2}.

\

Comparing \eqref{mi-p=1} with \eqref{mi-p=palpha} or \eqref{mi-p=infty-1} one sees that the positive radial solution has Morse index $1$ when $p$ is close to $1$, and greater than $1$ when $p$ is at the opposite end of the existence range (for $\a>0$).
\\
It is not hard to see that in dimension $N=2$ the solution with 2 nodal zones shares the same behaviour, for every $\a\ge 0$.
Indeed in this case formulas \eqref{mi-p=1} and \eqref{mi-p=1brutta}, respectively,  simplify into
\[
m(u_p)= 2  \left\lceil \frac{2+\a}{2}\beta \right\rceil 
\]
if $\alpha \neq \a_{n} = 2(n+1)/\beta -2$, or
\[
(2+\a)\beta \le  m(u_p)	\le (2+\a)\beta + 2
\]
if $\alpha = \a_{n}$ for some  for some integer $n$.
So remembering that $\lceil t \rceil < t+1$ we have
\[ m(u_p)	\le (2+\a)\beta + 2 \quad \text{ for } p\in (1, \bar p) \]
for every $\a\ge 0$. Furthermore the parameter $\beta$  turns out to be 
\begin{equation}\label{beta} 
 \beta\approx 2,\!305
 \end{equation}
as noticed in \cite{Ama}.
Therefore taking $q>p^{\star}$ we deduce from \eqref{mi-p=infty-2}, \eqref{mi-p=infty-2-brutta} that 
\begin{align*} 
m(u_q) -m(u_p) & \ge   2 \left\lceil\frac{2+\a}{2}\kappa\right\rceil +   2\left\lceil\frac{\a}{2} \right\rceil  - (2+\a)\beta-2
\intertext{and since clearly $ \lceil t \rceil \ge t $ we have}
& \ge  (2+\a)(\kappa -\beta) +\a -2 \ge 2(\kappa-\beta -1)  > 2.
\end{align*}

In higher dimensions,  the approximation of the parameters $\beta_i$ appearing in \eqref{mi-p=1} can be numerically performed after having chosen a specific value for $\alpha$, which fixes the baseline Bessel function ${\mathcal J}_{\frac{N-2}{2+\a}}$. To have an overall picture it can be useful to establish some estimate.
We report here the elementary proof of an estimate of the Bessel zeros that contributes to this aim. 

\begin{lemma}\label{betap=1-}
	For all $\beta>0$ and $i, m$ integers with $i< m$ we have 
	\begin{align}\label{est-bessel}  z_{i}(\beta+2(m-i))<z_{m}(\beta) .
	\end{align}
\end{lemma}

\begin{proof}
	It is known that the $m^{th}$ zero of $\jcal_{\b}$ lies in the $m^{th}$ nodal set of $\jcal_{\b+1}$, i.e.
	\begin{equation}\label{f1}
	z_{m-1}(\beta+1)<z_{m}(\b) < z_{m}(\b+1),
	\end{equation}
	which also implies $(-1)^{m} \jcal_{\b+1}(z_m(\b))<0$.
	On the other hand 
	\begin{equation}\label{f2}
	z_{m-2}(\b+2)<z_{m}(\b)<z_{m}(\b+2).
	\end{equation}
	Actually the first inequality is obtained by iterating \eqref{f1} and the second one follows since the the map  $\b\mapsto z_m(\b)$ is increasing.
	Hence the $m^{th}$ zero of $\jcal_{\b}$ can only belong to the $(m-1)^{th}$ or to the $m^{th}$ nodal set of $\jcal_{\b+2}$.
	But  by the three point recurrence relation
	\[ \jcal_{\b+2} (z_m(\b)) = \dfrac{2(\b+1)}{z_m(\b)}\jcal_{\b+1} (z_m(\b)) -\jcal_{\b} (z_m(\b))
	= \dfrac{2(\b+1)}{z_m(\b)}\jcal_{\b+1} (z_m(\b)) ,
	\]
	and therefore also $(-1)^{m} \jcal_{\b+2}(z_m(\b))<0$, which means that the $m^{th}$ zero of $\jcal_{\b}$ lies in the $m^{th}$ nodal set of $\jcal_{\b+2}$. In particular
	\begin{equation}\label{f3}
	z_{m-1}(\b+2)<z_{m}(\b).
	\end{equation}
	Applying iteratively  \eqref{f3} gives  the claim.
\end{proof}

Inserting the estimate \eqref{est-bessel} inside the formulas \eqref{mi-p=1} and \eqref{mi-p=1brutta} gives  the following
\begin{proposition}\label{prop:morsep=1-}
For every $\a\ge 0$ we have
\begin{equation}\label{unico-controllo}
\beta_i > \frac{N-2}{2+\a} + 2 (m-i) \ \text{ as } i=1,\dots m-1 
\end{equation}
and there is $\bar p =\bar p(\a)>1$ such that the Morse index of $u_p$  is estimated from below by
\begin{align} \label{morsep=1-}
m(u_p)& \ge    1 + \sum\limits_{i=1}^{m-1}\sum\limits_{j=0}^{[(2+\a)(m-i) ]} N_j  
\\ \label{morsep=1-bis}
&  = m+ \sum\limits_{k=1}^{m-1}(m-k) \sum\limits_{j=1+[(2+\a) (k-1)]}^{[(2+\a) k]} N_{j} 
\\ \label{morsep=1-ter}
& \ge m+ \sum\limits_{k=1}^{m-1}(m-k) \sum\limits_{j=1+(2+[\a])(k-1)}^{(2+[\a])k} N_{j} 
\end{align}
for $p\in(1, \bar p)$.
\end{proposition}
\begin{proof}
 \eqref{unico-controllo} is an immediate consequence of Lemma \ref{betap=1-}  since $\beta_m= \frac{N-2}{2+\a}$ and the map $\beta \mapsto z_i(\beta)$ is increasing.
In particular the index $J_i(p)$ appearing in \eqref{tag-2-H}  satisfy $J_i(p) > (2+\a)(m-i)$ in a right neighborhood of $p=1$, and 
	 \eqref{morsep=1-} follows.
	  \\
		Next,  
			\[  \begin{split}
		1 + \sum\limits_{i=1}^{m-1}\sum\limits_{j=0}^{[(2+\a)(m-i) ]} N_j = m + \sum\limits_{i=1}^{m-1}\sum\limits_{j=1}^{[(2+\a)(m-i) ]}  N_j
		= m + \sum\limits_{i=1}^{m-1}\sum\limits_{k=1}^{m-i}\sum\limits_{j=1+[(2+\a)(k-1)]}^{[(2+\a)k ]} N_j\\
		= m + \sum\limits_{k=1}^{m-1}\sum\limits_{i=1}^{m-k}\sum\limits_{j=1+[(2+\a)(k-1)]}^{[(2+\a)k ]} N_j
		= m + \sum\limits_{k=1}^{m-1}(m-k)\sum\limits_{j=1+[(2+\a)(k-1)]}^{[(2+\a)k ]} N_j ,
		\end{split}\]
			  which is \eqref{morsep=1-bis}.
		\\
		Moreover, as clearly $[(2+\a) k]\ge \left(2+[\a]\right) k$, we have
		\begin{align*}
		(m-k) \sum\limits_{j=1+[(2+\a)(k\!-\!1)]}^{[(2+\a) k]} N_j  = (m-k)\sum\limits_{j=1+[(2+\a)(k\!-\!1)]}^{(2+[\a])k} N_j  + (m-k) \sum\limits_{j=1+(2+[\a])k}^{[(2+\a)k] } N_j \\
		\underset{k'=k+1}{=} (m-k)\sum\limits_{j=1+[(2+\a)(k\!-\!1)]}^{(2+[\a])k} N_j  + (m-k'+1) \sum\limits_{j=1+(2+[\a])(k'-1)}^{[(2+\a)(k'-1)] } N_j  \\
		\ge (m-k)\sum\limits_{j=1+[(2+\a)(k\!-\!1)]}^{(2+[\a])k} N_j  + (m-k') \sum\limits_{j=1+(2+[\a])(k'-1)}^{[(2+\a)(k'-1)] } N_j 
		\end{align*}
		Hence
			\[  \begin{split} m(u_p)\ge m + \sum\limits_{k=1}^{m-1}(m-k)\sum\limits_{j=1+[(2+\a)(k-1)]}^{[(2+\a)k ]} N_j 
		\\	\ge   m + \sum\limits_{k=1}^{m-1}(m-k)\sum\limits_{j=1+[(2+\a)(k-1)]}^{(2+[\a])k } N_j 
			+ \sum\limits_{k=2}^{m-1}(m-k)\sum\limits_{j=1+(2+[\a])(k-1)}^{[(2+\a)(k-1)] } N_j \\
			= m + (m-1) \sum\limits_{j=1}^{2+[\a] } N_j  + \sum\limits_{k=2}^{m-2}(m-k)\sum\limits_{j=1+(2+[\a])(k-1)}^{(2+[\a])k } N_j ,
		\end{split}\]
		which is \eqref{morsep=1-ter}.
	\end{proof}

We therefore see that, in dimension $N\ge 3$, the Morse index of nodal radial solutions for $p$ close to $p_{\a}$ is smaller than the one  for $p$ is close to $1$.

	\begin{corollary}\label{cambio-morse-N>3}
	In dimension $N\ge 3$, for every value of $\a\ge 0$ and $m\ge 2$ there exist $1<\bar p< \bar q < p_{\a}$ such that 
	we have
	\[ m(u_p)  > m(u_q)  \qquad \text{ as } 1<p<\bar p \text{ and }  \bar q< q < p_{\a}. \]
\end{corollary}
\begin{proof}
	By \eqref{mi-p=palpha}  we know that for every $\a\ge 0$
	\[ m(u_q) \le m \sum\limits_{j=0}^{\left[\frac{2+\alpha}{2}\right]} N_j  \]
	as long as $q$ is near $p_{\a}$.
	So, thanks to the estimate \eqref{morsep=1-ter}, the claim follows after checking that  
	\begin{equation}\label{suff} \begin{split}
	h(m): = \sum\limits_{i=1}^{m-1}(m-i) \sum\limits_{j=1+(2+[\a])(i-1)}^{(2+[\a]) i} N_{j} - m \sum\limits_{j=1}^{1+\left[{\alpha}/{2}\right]} N_j 
	>0 .
	\end{split}\end{equation}
	\eqref{suff} can be proved by induction on the number of nodal zones $m\ge 2$,  taking advantage from the fact that in dimension $N\ge 3$ the multeplicity $N_j$ increases with $j$, i.e.
	\begin{equation}\label{mult-incr}
	N_{j+1} > N_j \quad \text{ as \ } \ j\ge 1.
	\end{equation}
	\\
	We first check 
	\[ h(2)=\sum\limits_{j=1}^{2+[\a]} N_{j} - 2 \sum\limits_{j=1}^{1+\left[{\alpha}/{2}\right]} N_j = \sum\limits_{j=2+[\a/2]}^{2+[\a]} N_{j} -  \sum\limits_{j=1}^{1+\left[{\alpha}/{2}\right]} N_j 
	>0 .\]
	When $\a\in[0,2)$, $[\a]\ge [\a/2]=0$ and we have $h(2)\ge N_2- N_1 >0$.
	\\
	Otherwise if $\alpha\ge 2$ then $[\a/2]\ge 1$ and  \eqref{mult-incr} yields
	\begin{align*} h(2) & >\left([\a]-[\a/2]\right) N_{2+[\a/2]} - [\a/2] N_{1+[\a/2]} 
	\intertext{and since $[\a]\ge 2 [\a/2]$ we have}
	& \ge  [\a/2] \left(N_{2+[\a/2]} - N_{1+[\a/2]} \right)>  0
	\end{align*}
	by using   \eqref{mult-incr} once more.
	\\
	After we take that $h(m)>0$ for some $m\ge 2$ and deduce that also $h(m+1)>0$.
	Let us compute
	\begin{align*}
	\sum\limits_{i=1}^{m}(m+1-i) \sum\limits_{j=1+(i-1)(2+[\a])}^{i(2+[\a])} N_{j} = \sum\limits_{i=1}^{m-1}(m-i) \sum\limits_{j=1+(i-1)(2+[\a])}^{i(2+[\a])} N_{j} + \sum\limits_{j=1}^{(2+[\a])m} N_{j} ,
	\end{align*}
	hence
	\[ h(m+1) = h(m) + \sum\limits_{j=1}^{(2+[\a])m} N_{j} - \sum\limits_{j=1}^{1+[\a/2]} N_{j} > \sum\limits_{j=2+[\a/2]}^{(2+[\a])m} N_{j} >0 ,\]
	and this concludes the proof.
\end{proof}

	In the next section we will see that the changes in the Morse index caused by the first singular eigenvalue $\nu_1(p)$ play a crucial role in establishing bifurcation results. Therefore the parameter $\beta_1$ deserves a special attention, in particular the integer number 
	\begin{equation}\label{n-def}
	n_{\a}^m : = \left\lceil \frac{(2+\a)\beta_1 -N}{2}\right\rceil ,
	\end{equation}
which is characterized by the double inequality
\begin{equation}\label{n-def-char}
z_1\left(\frac{ 2 n_{\a}^m+N-2}{2+\a}  \right) < z_m\left(\frac{N-2}{2+\a}  \right) \le  z_1\left(\frac{2n_{\a}^m+N}{2+\a} \right) .
\end{equation}
Once  that the dimension $N$, the exponent $\a$ and  the number of nodal zones $m$ have been fixed, the number $n^m_{\a}$ can be easily computed by using iteratively the function \texttt{Besselzero} in MathLab, for instance. Besides it is already known by \eqref{unico-controllo} that 
	\begin{equation}\label{n-unico-controllo} n_{\a}^m \ge 2(m-1) + [ \a(m-1)] .
	\end{equation} 
For $\a=0$ (Lane-Emden equation) there is numerical evidence that for every $N$
\[ z_1\left(2(m-1) + \frac{N-2}2   \right) < z_m\left(\frac{N-2}2  \right) \le  z_1\left(2(m -1) +\frac{N}2 \right) , \]
so that  $n_{0}^m = 2(m-1)$ indeed.

\section{Global bifurcation}\label{sec:bif}


Here we prove the bifurcation results stated in Section \ref{sec:stat}. 
It is well known that if $(p, u_p)$ is a bifurcation point in the curve $\mathcal S^m$, then the solution  $u_{p}$ has to be degenerate, which means that the linearized operator
$L_{p}$ defined in \eqref{linearized} has nontrivial kernel in $H^1_0(B)$, or equivalently  $\L=0$ is an eigenvalue for \eqref{standard-eig-prob}.  
In Section \ref{sec:prel} we have noticed that the Morse index changes within the interval $(1, p_{\a})$, so that degeneracy values do exist.
Besides we can not rely on any variational structure, since we aim to include also supercritical values of $p$, and bifurcation can be obtained only through an odd change of the Morse index. Hence  a better knowledge of the kernel of $L_p$ is needed.
By \cite[Theorem 1.3]{AG-sing-2} the radial solutions are radially nondegenerate, i.e.~the kernel of $L_{p}$ does not contain radial functions. Moreover, the degeneracy has been characterized in \cite[Proposition 1.5]{AG-sing-1} in terms of the eigenvalues $\nu_i(p)$ showing that
\begin{proposition}\label{prop:degeneracy}
Let 	$u_p$ be a radial solution to \eqref{H} with $m$ nodal zones. It  is degenerate if and only if
	\begin{equation}\label{non-radial-degeneracy-H}
	{\nu}_i(p) =  - \left(\frac{2}{2+\a}\right)^2 j (N-2+j) \qquad \mbox{ for some $i=1,\dots m$ and $j\ge 1$.}
	\end{equation}
	Besides any function in the kernel of $L_{p}$ can be written according to the decomposition formula 
	\begin{equation}\label{decomposition}
 \phi(x)= \psi_{i,p}(|x|^{\frac{2+\a}2})Y_j(x/|x|) ,\end{equation}	
 where $	\psi_{i,p}$ is an eigenfunction for \eqref{eigenvalue-problem} related to an eigenvalue $\nu_i(p)$ satisfying \eqref{non-radial-degeneracy-H}, and $Y_j$ stands for an eigenfunction of the eigenvalue $j (N-2+j)$ of the Laplace-Beltrami operator.
\end{proposition}
For a positive solution only the first eigenvalue $\nu_1(p)$ plays a role and one can manage to obtain an odd change in the Morse index by restricting the attention to the subspace of $O(N-1)$-invariant functions, as in \cite{AG14, FN17}.
For a nodal solution, instead,  the equality \eqref{non-radial-degeneracy-H} can hold  for different values of $i$ and $j$  and \eqref{decomposition} brings out that  the  kernel of $L_p$ has a complex structure.
This difficulty can be dealt with  by  turning to the notion of degree and index of fixed points in the positive cones introduced in  Section \ref{sec:stat}.

		Letting $T$ be the  operator
\[ T(p,v) :   (1,p_{\a})\times  C^{1,\gamma}_0(B)\longrightarrow  C^{1,\gamma}_0(B) , \quad  T(p,v)=(-\Delta)^{-1}\left(|v|^{p-1}v\right), \]
it is clear that $\mathcal S^m$  are curves of fixed points for $T$ and more generally $u$ solves \eqref{H} when $u=T(p,u)$.
Moreover 
minor variations on \cite[Lemmas 2.2, 3.1]{AG-bif} allow seeing that 

\begin{lemma} 
	The operator $T(p,\cdot)$ maps both  $X_n$ and $K_n$ into themselves.
	\end{lemma}

Denoting by $T'_{u}(p,\cdot)$ the Fr\'echet derivative of $T(p,\cdot)$ computed at $u$, we say that $u$  is an isolated fixed point for $T(p,\cdot )$ w.r.t. $X_n$ 
when $I-T'_{u}(p,\cdot)$ is invertible in $X_n$, which is assured by the nondegeneracy of $u$. Starting from the characterization of degeneracy in Proposition \ref{prop:degeneracy} one can see that radial solutions $u_p$ are  isolated fixed points, except at most a discrete set of $p$.  
It follows from a  general regularity result.

\begin{lemma}\label{lem:analiticita}
	The maps $p\mapsto \nu_i(p)$ are analytic in $p$. 
\end{lemma} 
We do not report the details of the proof.  For positive solutions to the H\'enon equation it has been proved in \cite[Proposition 4.1]{AG14}.
For sign changing solutions to the Lane-Emden equation in dimension $N\ge 3$ the proof is contained in \cite[Lemma 3.2]{DW}, and it has been adapted to the case $N=2$ in  \cite[Lemma 7.1]{GI}.

\

One can now compute the index of $u_p$ relative to the cone $K_n$, see \cite{D83}, which will be denote by $\mathrm{index}_{K_n} (p, u_p)$.
It is important to note that, also in the case of nodal solutions, the first singular eigenvalue determines by itself such index.

	\begin{lemma}\label{soloi=1}
		Let $p$ be such that $u_p$ is nondegenerate. Then 
			\[
		\mathrm{index}_{K_n}(p,u_p)=\left\{\begin{array}{ll}
		0 & \text{ if }\nu_1(p)< -\big(\frac{2}{2+\a}\big)^2 n (N-2+n), \\[.2cm]
		{\mathrm{deg}}_{X_n}(I-T(p,\,))=\pm 1
		& \text{ if }\nu_1(p)> -\big(\frac{2}{2+\a}\big)^2 n (N-2+n) . \end{array}\right.
		\]
	\end{lemma}
Here the symbol ${\mathrm{deg}}_{X_n}(I-T(p,\, ))$ stands for the Leray-Shauder degree of the operator $I-T(p,\, )$ restricted at $X_n$, computed in a neighborhood  of $(p, u_p)$ which does not contain nonradial solutions (this choice is possible since $u_p$ is nondegenerate by assumption).
	\begin{proof}
		Theorem 1 in \cite{D83} states that for isolated fixed points
		\[
		\mathrm{index}_{K_n}(p,u)= \begin{cases}
		0 &   \hbox{ if } T'(p,u)  \mbox{ has the property $\a$}, 
		\\[.2cm]	{\mathrm{deg}}_{X_n}(I-T(p,\, ))=\pm 1 & \mbox{ otherwise.}
		\end{cases}\]
	In this way, the prove reduces to show that the so-called {\em property $\alpha$} holds if and only $\nu_1(p)+\big(\frac{2}{2+\a}\big)^2 n (N-2+n)<0$.
		Several characterizations of the {\em property $\alpha$}  are provided in Lemma 3 and the following Remark in  \cite{D83}.
		To state the one which will be used here we need the sets
		\[\begin{array}{ll}
		W^+ := &\!\! \!\!  \{v\in X_n \ u_p+\gamma v\in K_n\ \text{ for some }\gamma>0\}, \\
		W^0 := &\!\! \!\! \{v\in W^+_{u_p}\ : -v\in  W^+_{u_p}\}, \\
		V   &\!\! \!\! \mbox{ the orthogonal (in the $H^1_0$ sense) complement to $W^0$ in $X_n$.}
		\end{array}\]
		Notice that the functions in $W^0$ do not depend by the angle $\theta$.
		Next,    $T'$ has the {\em property $\a$} if 
		there exists $t\in (0,1)$ such that the problem  
		\begin{equation}\label{property-alpha}
		\begin{cases} 
		-\Delta v = tp |x|^{\a}|u_p|^{p-1}v & \text{ in } B , \\
		v\in \overline{W^+}\setminus W^0  & 
		\end{cases}\end{equation}
		has a solution.
		We follow the proof of \cite[Theorem 1]{D92} and look at the family of eigenvalue problems
		\begin{equation}\label{lambdat}
		\begin{cases} -\Delta v -t p |x|^{\a} |u_p|^{p-1}v = \Lambda v & \text{ in } B , \\
		v\in V & \end{cases}
		\end{equation}
		and  let $\Lambda_t$ be its first eigenvalue. 
		When $t=0$ \eqref{lambdat} reduces to an eigenvalue problem for the Laplacian and certainly $\L_0>0$.
		When $t=1$, instead,  \eqref{lambdat} gives back the eigenvalue problem \eqref{standard-eig-prob}, but only eigenfunctions in $V$ matter.
		Furthermore the variational characterization yields that the first eigenvalue $\Lambda_t$ is strictly decreasing w.r.t.~$t$.
		\\
		If $T'$ has the property $\alpha$, then $\Lambda_t\le 0$ for some $t<1$ and therefore $\Lambda_1<0$. This in turn means that the eigenvalue problem \eqref{standard-eig-prob}  has a negative eigenvalue with related eigenfunction in $V$ and then $\nu_i(p)+ \big(\frac{2}{2+\a}\big)^2 j (N-2+j)<0$ for some $i=1,\dots m$ and $j$ such that the related spherical harmonic belongs to $V$, by the characterization in \cite[Theorem 1.4]{AG-sing-1}.
Taking advantage from  the description of the spherical harmonics given in the proof of Theorem 1.1 in \cite{AG-bif}, one sees that  $j$ must be a multiple of $n$ and so, in particular, $\nu_1(p)+ \big(\frac{2}{2+\a}\big)^2 n (N-2+n)<0$.
		\\
		On the other hand if $\nu_1(p)+\big(\frac{2}{2+\a}\big)^2 n (N-2+n)<0$ we let $\psi$ be the first radial eigenfunction for \eqref{eigenvalue-problem} and $Y_n$ the spherical harmonic related to $n (N-2+n)$ belonging $V$ (which does exist for what we have said before). Now $v (r,\theta,\phi)= \psi(r^{\frac{2+\a}{2}}) \, Y_n(\theta, \phi)$ is in $V$ and an easy computation shows that 
		\begin{align*}
		\int_B\left(|\nabla v|^2 - p|x|^{\a}|u_p|^{p-1} v^2\right) dx = 
		\\		=	\int_0^1 r^{N-1} \left[ \left(\frac{d}{dr}\psi\big(r^{\frac{2+\a}{2}}\big)\right)^2  - p r^{\a}  |u_p|^{p-1} \left(\psi(r^{\frac{2+\a}{2}}) \right)^2 \right] dr 
		\int_{\mathbb S_{N\!-\!1}} Y_n^2  d\sigma(\theta,\varphi) \\
	+ 	\int_0^1 r^{N-3} \left(\psi\big(r^{\frac{2+\a}{2}}\big) \right)^2  dr  
	\int_{\mathbb S_{N\!-\!1}}  | \nabla Y_n|^2 d\sigma(\theta,\varphi) 
		\\
			= \int_0^1 r^{N-1+\a} \left[ \left(\frac{2+\a}{2}\right)^2 \left(\psi'(r^{\frac{2+\a}{2}}) \right)^2  - p   |u_p|^{p-1} \left(\psi(r^{\frac{2+\a}{2}}) \right)^2 \right] dr
		\int_{\mathbb S_{N\!-\!1}} Y_n^2  d\sigma(\theta,\varphi) \\
		+ 	\int_0^1 r^{N-3} \left(\psi\big(r^{\frac{2+\a}{2}}\big) \right)^2  dr  
	\int_{\mathbb S_{N\!-\!1}}  | \nabla Y_n|^2 d\sigma(\theta,\varphi) 
			\intertext{using the change of variable $t=r{\frac{2+\a}{2}}$ and the notation in \eqref{ap}}
			 =	\frac{2+\a}{2} \int_0^1 t^{M-1} \left[ \left(\psi'(t) \right)^2  - a_p(t) \left(\psi(t) \right)^2 \right] dt
		\int_{\mathbb S_{N\!-\!1}}Y_n^2  d\sigma(\theta,\varphi) \\
				+ 	\frac2{2+\a}	\int_0^1 t^{M-3} \psi^2(t) \, dt 
			\int_{\mathbb S_{N\!-\!1}}  | \nabla Y_n|^2 d\sigma(\theta,\varphi) 
		\intertext{and as $\psi$ solves \eqref{eigenvalue-problem} and $Y_n$ is a spherical harmonics we get}
		 =   \frac{2+\a}{2} \nu_1(p)\int_0^1 t^{M-3}  \psi^2(t) \,  dt
		 \int_{\mathbb S_{N\!-\!1}} Y_n^2  d\sigma(\theta,\varphi) \\
		 + \frac{2}{2+\a}	 n (N-2+n)		\int_0^1 t^{M-3} \psi^2(t) \, dt 
		 \int_{\mathbb S_{N\!-\!1}} Y_n^2 d\sigma(\theta,\varphi) 
			\\	=  \frac{2+\a}{2} \left( \nu_1(p)+\Big(\frac{2}{2+\a}\Big)^2 n (N-2+n)\right)\int_0^1 t^{M-3}  \psi^2(t) \, dt
		\int_{\mathbb S_{N\!-\!1}} Y_n^2  d\sigma(\theta,\varphi) \\
		=  \left( \nu_1(p)+\Big(\frac{2}{2+\a}\Big)^2 n (N-2+n)\right) \int_B \frac{v^2}{|x|^2} dx  <0 .
			 	\end{align*}
		Hence the first eigenvalue $\L_1$ is negative, and since $\L_0>0$ there exists $t\in(0,1)$ such that $\Lambda_t=0$, which means that $T'$ has the property $\alpha$.
	\end{proof}

Relying on Lemma \ref{soloi=1}  one can see that a sufficient condition for bifurcation is 
		\begin{equation}\label{lem:cedelbuono}	\left( \lim\limits_{p\to 1}\nu_1(p)+\Big(\frac{2}{2\!+\!\a}\Big)^2 n(N\!-\!2\!+\!n)\right) \left(  \lim\limits_{p\to p_{\a}}\nu_1(p)+\Big(\frac{2}{2\!+\!\a}\Big)^2 n(N\!-\!2\!+\!n)\right)< 0,	
		\end{equation}
	for some integer $n$.


\begin{proposition}\label{prop:bif} 
	 If $n$ is an integer which fulfills \eqref{lem:cedelbuono}, then there exists at least one $p_n \in (1, p_{\a})$ such that $(p_n, u_{p_n})$ is a nonradial bifurcation point and  
 the branch $\mathcal C_n$ defined according to \eqref{cn} is global, in the sense that it contains a sequence $(p_k,u_k)$ with
\begin{enumerate}[i)]
	\item either $\|u_k\|_{C^{1,\gamma}(B)}\to +\infty$,
	\item or $p_k\to p_{\a}$, 
	\item or $p_k\to 1$.
	\end{enumerate}
\end{proposition}
\begin{proof}
Under assumption \eqref{lem:cedelbuono}, and thanks to Lemma \ref{lem:analiticita}, there exists at least one (and an odd number of)  $\bar p\in (1, p_{\a})$ and $\delta>0$ such that 
\[\begin{array}{c}
{\nu}_1(\bar p) =  - \left(\frac{2}{2+\a}\right)^2 n (N-2+n) , \\
\left({\nu}_1(\bar p\!-\!\delta) + \big(\frac{2}{2+\a}\big)^2 n (N\!-\!2\!+\!n) \right) \left({\nu}_1(\bar p\!+\!\delta) + \big(\frac{2}{2+\a}\big)^2 n(N\!-\!2\!+\!n) \right) <0 , 
\\
\nu_{i}(p)\neq - \left(\frac{2}{2+\a}\right)^2 j (N-2+j)  
\end{array} 	\]
for every $i=1,\dots m$, $j\ge 0$, and $p\in(\bar p-\delta,\bar p+\delta)$, $p\neq \bar p$.
\\
Lemma \ref{soloi=1} then implies that the Leray Schauder degree in the cone $K_n$ changes and the remaining of the proof follows as in \cite[Theorem 1.2]{AG-bif}.
See also \cite{Gla-glob}, where a more detailed proof is given in the case of positive solutions.
\end{proof}

We are now ready to prove the bifurcation results stated in Section \ref{sec:stat}.
Concerning positive solutions, we have already pointed out  that  the Morse index near at $p_{\a}$ is strictly greater than the one near at $1$, for every $\a>0$.
Since only the first eigenvalue $\nu_1(p)$ is negative and gives a contribution to the Morse index, it is clear that there exists at least one value of $n$ such that  \eqref{lem:cedelbuono} holds. Let us complete the proof of Theorem \ref{teo:bif-H-1}.

\begin{proof}[Proof of Theorem \ref{teo:bif-H-1}]
First, we  check that \eqref{lem:cedelbuono} is fulfilled for every integer $n=1,\dots \lceil{\a}/{2}\rceil$.
Recalling that 	$  \lim\limits_{p\to 1}\nu_1(p) =0$ by \eqref{nup=1} (since for positive solutions $\beta_1=\frac{N-2}{2+\a}$), it is equivalent to see that 
\[\lim\limits_{p\to p_{\a}}\nu_1(p) < - \Big(\frac{2}{2+\a}\Big)^2 n(N-2+n)\]  
for $1\le n<\frac{2+\a}{2}$, i.e.
\[ \lim\limits_{p\to p_{\a}}\nu_1(p) \ge  -  \frac{2N-2+\a}{2+\a} .\]
But \eqref{nu-p=palpha} and \eqref{nu-p=infty-1} state that equality holds in any dimension $N\ge 2$.

Therefore Proposition \ref{prop:bif} gives the first part of the claim. 
As for property {\it i)}, every branch $\mathcal C_n$ must be composed of nonnegative solutions by continuity, so that maximum principle ensures that or they are  positive, or they are identically zero. But this last occurrence is not allowed since the trivial solution is isolated.
Besides the same Proposition \ref{prop:bif} states that $\mathcal C_n$ contains a sequence $(p_k, u_{k})$ such that either $\|u_k\|\to \infty$, or $p_k\to p_{\a}$, or $p_k\to 1$.
Moreover the occurrence $p_k\to 1$ is forbidden by the uniqueness of positive solutions for $p$ close to $1$ in \cite[Theorem 3.1]{AG-bif} (which can be easily extended also to dimension $N=2$).
\\
As for the possible  intersection between two branches $\mathcal C_n$ and $\mathcal C_{n'}$, it has to be  composed by $(p, v)$  such that $v\in K_n\cap K_{n' }$ is a positive solution to \eqref{H}.
In dimension $N=2$ $K_n\cap K_{n' }$ reduces to radial function, and therefore $v=u_p$ is a radial positive solution to \eqref{H}, which has to be degenerate and therefore isolated.
In dimension $N\ge 3$, instead, $K_n\cap K_{n' }$ contains also functions which are nonradial, but do not depend by the angle $\theta$.
%
	\end{proof}


\

After we deal with bifurcation from nodal solutions, and we begin by examining the planar case.

\begin{proof}[Proof of Theorem \ref{teo:bif-H-m-N=2}]
First, we compute the values of the integer $n$ for which  \eqref{lem:cedelbuono} holds.  
Thanks to  \eqref{nup=1} and \eqref{nu-p=infty-2}, it  means that
\[ -\beta^2= \lim\limits_{p\to 1} \nu_1(p) > -\left(\frac{2n}{2+\a}\right)^2>  \lim\limits_{p\to \infty} \nu_1(p)=-\kappa^2,\] 
which is clearly equivalent to   $\frac{2+\a}{2}\beta <n<\frac{2+\a}{2}\kappa$, i.e. $n= \left[\frac{2+\a}{2}\beta +1 \right] , \dots \left\lceil\frac{2+\a}{2}\kappa-1\right\rceil$.
Here $\beta$ and $\kappa$ are respectively given by \eqref{beta} and \eqref{nu-p=infty-2}.
\\
So Proposition \ref{prop:bif} yields that, for any of such values of $n$, there exist a nonradial bifurcating point $(p_n, u^2_{p_n})$ and a bifurcating branch in $K_n$. 
The local property mentioned {\it i)} is a plain consequence of the continuity of the branch.
As for property {\it ii)},  Proposition \ref{prop:bif}  states that the branch $\mathcal C_n$ contains a sequence $(p_k, u_{k})$ such that either $\|u_k\|_{C^{1,\gamma}}\to \infty$, or $p_k\to \infty$, or $p_k\to 1$. If $p_k\to 1$ but $\|u_k\|_{C^{1,\gamma}}$ stays bounded,  then \cite[Lemma 2.1]{Ama} ensures that $u_k$ converges to an eigenfunction of \eqref{prima-autof-weight}.
Coming to  property {\it iii)}, since $K_n\cap K_{n'}$ is the set of radial functions, then  the intersection point between  two branches $\mathcal C^2_n$ and $\mathcal C^2_{n'}$  should be another nonradial bifurcation point  $(p, u^m_p)$. Here possibly $m\neq 2$, because  the number of nodal zones of the bifurcating solutions could become larger that $2$, far from the bifurcation point. If this happens, then the very definition of $\mathcal C^2_n$ yields that $\mathcal S^m \subset \mathcal C^2_n$ and this concludes the proof.
\end{proof}

The proof of the bifurcation in higher dimensions is quite similar.
\begin{proof}[Proof of Theorem \ref{teo:bif-H-m}]
Let us check that property \eqref{lem:cedelbuono} holds as long as 
$ \frac{2+\a}{2} < n < \frac{2+\a}{2} \beta_1 - \frac{N-2}{2}$, where $\beta_1$ is defined in \eqref{betai-def}.
We have already noticed that now the first  singular eigenvalue  close to $p=p_{\a}$ is greater than close to $p=1$, therefore it is needed that
\[ \lim\limits_{p\to p_{\a}} \nu_1(p) > -\left(\frac{2}{2+\a}\right)^2n(N-2+n) > \lim\limits_{p\to 1} \nu_1(p)  .\]
	By \eqref{nup=1} and \eqref{nu-p=palpha}  it means that  
	\[ \frac{2N-2+\a}{2+\a} < \left(\frac{2}{2+\a}\right)^2 n(N-2+n) < \b_1^2 - \left(\frac{N-2}{2+\a}\right)^2  ,\]
which can be rearranged into
		\[ 1 + 2\frac{N-2}{2+\a}< \left(\frac{2n}{2+\a}\right)^2 + 2 \frac{2n}{2+\a}\frac{N-2}{2+\a} < \b_1^2 - \left(\frac{N-2}{2+\a}\right)^2 ,\]
		which in turn, after adding the term $\Big(\frac{N-2}{2+\a}\Big)^2$ to every member and extracting square roots, becomes 
			\[ \frac{N+\a}{2+\a} < \frac{N-2+2n}{2+\a}  < \b_1  ,\]
		i.e.
			\[ \frac{2+\a}{2}< n<  \frac{2+\a}{2} \beta_1 - \frac{N-2}{2}  \]
		as claimed.	It means that \eqref{lem:cedelbuono} is fulfilled by $n= 2 + \left[\frac{\a}{2}\right], \dots n_\a^m$, where $n_\a^m$ is defined in \eqref{n-def}.
 Remembering that  $n_\a^m \ge 2(m-1)+ [ \a (m-1)$, see \eqref{n-unico-controllo}, one can see that 
 \eqref{lem:cedelbuono} holds for at least $2m-3 + [\a(m-1)] - \left[\frac{\a}{2}\right] $ different values of $n$.

	Eventually the conclusion follows by Proposition \ref{prop:bif}	, arguing as in the proof of Theorem \ref{teo:bif-H-m-N=2}. The only difference stands in the possible overlapping between branches, i.e. property {\it iii)}, which in higher dimension can also contain nonradial solutions which do not depend by the angle $\theta$. 
\end{proof}


\begin{thebibliography}{99}
	
\bibitem{AP} {\sc A. Aftalion, F.  Pacella},
Qualitative properties of nodal solutions of semilinear elliptic equations in radially symmetric domains 
 {\it Comptes Rendus Mathematique}, 339/5 (2004),  339-344. DOI: 10.1016/j.crma.2004.07.004 

\bibitem{Ama} {\sc A.L. Amadori,}  On the asymptotically linear H\'enon problem, (2019) Preprint arXiv:1906.00433

\bibitem{AG14} {\sc A.L. Amadori, F. Gladiali}, Bifurcation and symmetry breaking for the H\'enon equation, {\em Advances in Differential Equations} {\bf 19} (2014) n.7-8,  755-782, http://projecteuclid.org/euclid.ade/1399395725

\bibitem{AG-bif}{\sc A.L.~Amadori, F.~Gladiali},  Nonradial sign changing solutions to Lane-Emden problem in an annulus {\em Nonlinear Analysis, Theory, Methods and Applications} 
155 (2017), 294-305.

\bibitem{AG-sing-1} {\sc A.L. Amadori, F. Gladiali}, On a singular eigenvalue problem and its applications in computing the Morse index of solutions to semilinear PDE's (2018) arXiv:1805.04321

\bibitem{AG-sing-2} {\sc A.L. Amadori, F. Gladiali}, On a singular eigenvalue problem and its applications in computing the Morse index of solutions to semilinear PDE's, Part II (2019) arXiv:1906.00368

\bibitem{AG-N>3} {\sc A.L. Amadori, F. Gladiali},  Asymptotic profile and Morse index of nodal radial solutions to the H\'enon problem, (2018) arXiv:1810.11046, to appear on {\em Calculus of Variations and Partial Differential Equations}

\bibitem{AG-N=2}{\sc A.L. Amadori, F. Gladiali}, The H\'enon problem with large exponent in the disc, (2019)  arXiv:1904.05907

\bibitem{BS}{\sc M. Badiale, E. Serra}, Multiplicity results for the supercritical H\'enon equation 
{\em Advanced Nonlinear Studies}, 4/4 (2004), 453-467, DOI: 10.1515/ans-2004-0406 

\bibitem{BDP} {\sc T. Bartsch, T. D'Aprile, A. Pistoia}, 
On the profile of sign-changing solutions of an almost critical problem in the ball
{\em Bulletin of the London Mathematical Society}, 45/6 (2013), 1246-1258. DOI: 10.1112/blms/bdt061

\bibitem{BMP} {\sc T. Bartsch, A.M. Micheletti, A. Pistoia}, 
On the existence and the profile of nodal solutions of elliptic equations involving critical growth
{\em Calculus of Variations and Partial Differential Equations}, 26/3 (2006), 265-282. DOI: 10.1007/s00526-006-0004-6


\bibitem{BW}{\sc T. Bartsch, T. Weth}, A note on additional properties of sign changing solutions to superlinear elliptic equations,{\em Topological Methods in Nonlinear Analysis} {\bf 22} (2003), 1-14.


\bibitem{BWi} {\sc T. Bartsch, M. Willem},
Infinitely many radial solutions of a semilinear elliptic problem on $\R_n$
{\em Archive for Rational Mechanics and Analysis},  124/3 (1993), 261-276. DOI: 10.1007/BF00953069








 \bibitem{D83}{\sc E.N. Dancer}, On the indices of fixed points of mappings in cones and applications.
{\em Journal of Mathematical Analysis and Applications}  91/1 (1983), 131-151. DOI:10.1016/0022-247X(83)90098-7.

 \bibitem{D92}{\sc E.N. Dancer}, 
Global breaking of symmetry of positive solutions on two-dimensional annuli. {\em Differential Integral Equations} {\bf 5/4} (1992), 903-913.

\bibitem{DW}{\sc E.N. Dancer, J. Wei}, Sign-changing solutions for supercritical elliptic problems in domains with small holes. {\em Manuscripta Math.} {\bf 123} (2007), 493--511.

\bibitem{DIP-N>3}  {\sc F. De Marchis, I. Ianni, F. Pacella},
A Morse index formula for radial solutions of Lane-Emden problems
(2017) {\em Advances in Mathematics}, {\bf 322}, pp. 682-737. DOI: 10.1016/j.aim.2017.10.026 


\bibitem{DIP-N=2}  {\sc F. De Marchis, I. Ianni, F. Pacella}, Exact Morse index computation for nodal radial solutions of Lane-Emden problems  (2017) {\em Mathematische Annalen}, {\bf 367} (1-2), pp. 185-227. DOI: 10.1007/s00208-016-1381-6



\bibitem{EMP} {\sc P. Esposito, M. Musso, A. Pistoia},
On the existence and profile of nodal solutions for a two-dimensional elliptic problem with large exponent in nonlinearity. (2007)
{\em Proceedings of the London Mathematical Society}, 94/2, pp. 497-519. DOI: 10.1112/plms/pdl020

\bibitem{EPW} {\sc P. Esposito, A. Pistoia, J. Wei}, 
Concentrating solutions for the H\'enon equation in $\mathbb{R}^2$. (2006) Journal d'Analyse Mathematique, 100, pp. 249-280. DOI: 10.1007/BF02916763

\bibitem{FN17} {\sc P. Figueroa, S. Neves}, 
{Nonradial solutions for the H\'enon equation close to the threshold}, {\em Advanced Nonlinear Studies} (2019). DOI: 10.1515/ans-2019-2052

\bibitem{GNN} {\sc B. Gidas, W.M. Ni,  L. Nirenberg}, Symmetry and related properties via the maximum principle. {\em Comm. Math. Phys.} {\bf 68} (1979), 209- 243.



\bibitem{Gla-glob} {\sc F. Gladiali}, A global bifurcation result for a  semilinear elliptic equation, {\em  Journ. Math. Anal.  Appl.} {\bf 369} (2010), 306-311.

\bibitem{GGN}  {\sc F. Gladiali, M. Grossi, S.L.N. Neves}, Nonradial solutions for the H\'enon equation in $\R_n$, {\em Advances in Mathematics}, {\bf 249} (2013), 1-36, doi:10.1016/j.aim.2013.07.022.

\bibitem{GI}{\sc F. Gladiali, I. Ianni}, Quasiradial nodal solutions for the Lane-Emden problem in the ball, (2017) arXiv:1709.03315


\bibitem{HCZ} {\sc J. Hao, X. Chen, Y. Zhang},
Infinitely many spike solutions for the Hénon equation with critical growth,
{\em Journal of Differential Equations},
259/9 (2015),
 4924-4946. DOI:10.1016/j.jde.2015.06.015.
 
\bibitem{H} {\sc M. H\'enon}, Numerical experiments on the stability oh spherical stellar systems. {\em Astronom. Astrophys.} {\bf 24} (1973), 229-238.




\bibitem{IS} {\sc I. Ianni, A. Saldana}, Sharp asymptotic behavior of radial solutions of some planar semilinear elliptic problems, (2019) 	arXiv:1908.10503


\bibitem{KW} {\sc J. K\"ubler, T. Weth} Spectral asymptotics of radial solutions and nonradial bifurcation for the H\'enon equation, (2019) arXiv:1901.00453


\bibitem{Ni} {\sc W. M. Ni}, A Nonlinear Dirichlet Problem on the Unit Ball and Its Applications. {\em Indiana Univ. Math. J.} {\bf 31} (1982), 801-807.

\bibitem{NN} {\sc  W.M. Ni, R. D. Nussbaum}, Uniqueness and nonuniqueness for positive radial solutions of $\Delta u+f(u,r)=0$, {\em Comm. Pure Appl. Math.} {\bf 38} (1985), 67-108.

\bibitem{P} {\sc S.J. Peng}, 
Multiple boundary concentrating solutions to dirichlet problem of H\'enon equation
{\em Acta Mathematicae Applicatae Sinica}, 22/1 (2006), 137-162. DOI: 10.1007/s10255-005-0293-0

\bibitem{PS} {\sc A. Pistoia, E. Serra}, Multi-peak solutions for the H\'enon equation with slightly subcritical growth{\em Mathematische Zeitschrift}, 256/1 (2007) ,  75-97. DOI: 10.1007/s00209-006-0060-9

\bibitem{S} { {\sc E. Serra}, Non-radial positive solutions for the H\'enon equation with critical growth, {\em  Calc. Var. Partial Differ. Equ.} {\bf 23}, 301–326 (2005)}

\bibitem{SSW} {\sc D. Smets, , J. Su, M. Willem}, Non-radial ground states for the H\'enon equation. {\em Commun. Contemp. Math.} {\bf 4} (2002), 467-480.


\bibitem{ZY} {\sc Y.-B. Zhang, H.-T. Yang}, 
Multi-peak nodal solutions for a two-dimensional elliptic problem with large exponent in weighted nonlinearity
{\em Acta Mathematicae Applicatae Sinica}, 31/1 (2015),  261-276. DOI: 10.1007/s10255-015-0465-5

\end{thebibliography}
\end{document}